\documentclass[psamsfonts]{amsart}

\usepackage{amsmath, amssymb, amsthm}

\usepackage{amssymb,amsfonts}
\usepackage[all,arc]{xy}
\usepackage{enumerate}
\usepackage{mathrsfs}
\usepackage{hyperref}
\usepackage{graphicx}
\usepackage[pdftex]{pict2e}
\usepackage[dvipsnames]{xcolor}
\usepackage{ circuitikz }
\usepackage{tikz}
\usetikzlibrary{positioning}
\graphicspath{ {./Images/} }

\usepackage{tikz-cd}

\usepackage{calc}
\usepackage{enumitem}
\SetLabelAlign{parright}{\parbox[t]{\labelwidth}{\raggedleft#1}}
\usepackage{mdwlist}
\usepackage{array}
\usepackage{longtable}
\hypersetup{
hypertex=true,
colorlinks=true,
linkcolor=blue,
anchorcolor=blue,
citecolor=blue
}

\newcommand{\F}{\mathbb F}

\newcommand{\N}{\mathbb N}

\newcommand{\Ext}{\mathrm {Ext}}

\newcommand{\RNum}[1]{\uppercase\expandafter{\romannumeral #1\relax}}

\newtheorem{theorem}{Theorem}[section]
\newtheorem{corollary}[theorem]{Corollary}
\newtheorem{proposition}[theorem]{Proposition}

\theoremstyle{definition}
\newtheorem{definition}[theorem]{Definition}

\newtheorem{notn}[theorem]{Notation}

\theoremstyle{remark}
\newtheorem{remark}[theorem]{Remark}

\makeatletter
\let\c@equation\c@theorem
\makeatother
\numberwithin{equation}{section}

\begin{document}
\author{Kaixu Zhang, Dongming Zhang}
\title{Computations of classical Mahowald invariants at prime 2}
\maketitle

\begin{abstract}
We review the definition of Mahowald invariants and discuss the computational method described by Behrens\cite{behrens2005rootinvariantsadamsspectral}. Then we examine the relationship between the algebraic Mahowald invariants and the \( E \)-filtered Mahowald invariants, and compute the Mahowald invariants for some elements up to the 26-stem.
\end{abstract}

\tableofcontents
\newpage
\section{Introduction}
Let \( P^{\infty}_k \) denote the Thom spectrum \( \text{Th}(k\gamma \rightarrow \mathbb{R}P^{\infty}) \) associated with the \( k \)-fold sum of the tautological line bundle \( \gamma \) over the real projective space \( \mathbb{R}P^{\infty} \). For positive \( k \), this spectrum is equivalent to \( \mathbb{R}P^{\infty}/\mathbb{R}P^{k-1} \), where the cells of the real projective spectrum below dimension \( k \) are collapsed to a point.

The inclusion of $k\gamma$ into $(k+1)\gamma$ induces the map $ P^{\infty}_k\rightarrow  P^{\infty}_{k+1}$, and $ P^{\infty}_{-\infty}$ is the homotopy limit holim$_k P^{\infty}_{k}$. By Lin's theorem (see \cite{MR556925}), we have the 2-complete equivalence $S^{-1}\simeq\  P^{\infty}_{-\infty}$.

\begin{definition}
Let \( \alpha \) be an element of the \( n \)-th 2-primary stable homotopy group \( \pi_n(S^0) \). The Mahowald invariant (also called the root invariant) of \( \alpha \) is the coset \( M(\alpha) \) in the stable homotopy group of spheres such that the following diagram commutes:
\[\begin{tikzcd}
S^{n-1} \arrow[d, "{\alpha}"'] \arrow[r, "M({\alpha})", dashed] & S^{-N} \arrow[d] \\
S^{-1}\simeq P^{\infty}_{-\infty} \arrow[r]                     & P^{\infty}_{-N} 
\end{tikzcd}\]
where $N>1$ is minimal such that the left lower composition is nontrivial.
\end{definition}

In \cite{MR806016}, Jones showed the lower bound $M|(\alpha)|\geq2|\alpha|$ by employing a geometric interpretation of the Mahowald invariant $M(\alpha)$ based on $C_2$-equivariant stable homotopy theory. In \cite{MR1241877}, Mahowald and Ravenel defined algebraic Mahowald invariants $M_{alg}(\alpha)$ and discussed the relations between the homotopy Mahowald invariants and the algebraic Mahowald invariants. They proposed the conjecture that the Mahowald invariant converts
$\nu_n$–periodic families to $\nu_{n+1}$–periodic families.
In \cite{behrens2005rootinvariantsadamsspectral}, Behrens defined $E$-root invariants and filtered Mahowald invariants, and provided a computational method by excluding the possible candidates.

\hspace*{\fill}

Combining the method described by Behrens and the results on algebraic Mahowald invariants computed by Bruner\cite{MR1642887}, we compute the Mahowald invariants of all elements up to 26-stem with six exceptions.

\begin{theorem}
   The Mahowald invariants are determined for all elements in the stable homotopy groups of spheres up to the 26-stem, with the exception of the six elements $\eta_4,\nu_4,\bar{\sigma},\{P^2h_2\},4\bar{\kappa}$ and $4\nu\bar{\kappa}$.
\end{theorem}

\begin{longtable}{lllll}

\caption{Mahowald invariant}\label{table}\\

\hline
\textbf{Stem} & \textbf{Elements} &\textbf{$M_{alg}(\alpha)$}& \textbf{$M(\alpha)$} &  \textbf{Proof} \\
\hline 
\endfirsthead
\multicolumn{5}{l}%
{{\bfseries \tablename\ \thetable{} -- continued from previous page}} \\
\hline
\textbf{Stem} & \textbf{Elements} &\textbf{$M_{alg}(\alpha)$}&    \textbf{$M(\alpha)$} &\textbf{Proof} \\
\hline 
\endhead
\hline \multicolumn{5}{l}{{Continued on next page}} \\ \hline
\endfoot
\hline \hline
\endlastfoot

 1 & $\eta$ &  $h_2$ & $\nu$ & \cite{Behrens_2007} \\
 2 & $\eta^2$ & $h_2^2$ & $\nu^2$  & \cite{Behrens_2007} \\
 3& $\nu$ & $h_3$ & $\sigma$  &\cite{Behrens_2007} \\
 & $2\nu$ & $h_1h_3$ & $\eta\sigma$  & \cite{Behrens_2007} \\
   &   $4\nu$     &  $h_1^2h_3$ & $\eta^2\sigma$        & \cite{Behrens_2007}   \\
 6  &  $\nu^2$    & $h_3^2$ &   $\sigma^2$        & \cite{Behrens_2007}  \\
 7  &  $\sigma$      & $h_4$ & $\sigma^2$         &   \cite{Behrens_2007}  \\
  &     $2\sigma$   &  $h_1h_4$ & $\eta_4$          &  \cite{Behrens_2007}   \\
  &   $4\sigma$     & $h_1^2h_4$ &  $\eta\eta_4$      & \cite{Behrens_2007}\\
 &   $8\sigma$  &  $h_1^3h_4$ & $\eta^2\eta_4$     &  \cite{Behrens_2007}   \\
 8  &   $\eta\sigma$   &  $h_2h_4$ &  $\nu_4$            &   \cite{Behrens_2007}  \\
  &    $\epsilon$    &  $c_1$ &  $\bar{\sigma}$               &   Proposition \ref{easyones}  \\
 9  &   $\eta\varepsilon$    &  $h_2c_1$ &   $\nu\bar{\sigma}$          & Proposition \ref{easyones} \\
  &   $\eta^2\sigma$     &  $h_2^2h_4$ &  $\nu\nu_4$               &  \cite{Behrens_2007}  \\
  &   $\{Ph_1\}$   & $h_2g$    &    $\nu\bar{\kappa}$               &   \cite{Behrens_2007}  \\
10 &   $\{Ph_1^2\}$    & $d_0^2$   &     $\kappa^2$               &  \cite{Behrens_2007}   \\
11  &   $\{Ph_2\}$   & $h_2^2g$   &     $\nu^2\bar{\kappa}$                &  \cite{Behrens_2007} \\
  &   $\{Ph_2h_0\}$   & $q$ &  $\{q\}$   &           \cite{Behrens_2007}   \\
  &   $\{Ph_1^3\}$ & $h_1q$ & $\{h_1q\}$   &        \cite{Behrens_2007}   \\
 14  &  $\sigma^2$  & $h_4^2$ &    $\theta_4$                  &    Proposition \ref{easyones} \\
   &  $\kappa$  & $d_1$ &    $\kappa_1$                 &   Proposition \ref{easyones}   \\
15 &   $\rho_{15}$  &  $h_1^3h_5$ & $\eta^2\eta_5$          &   Proposition \ref{easyones}  \\
   &    $2\rho_{15}$    &  $h_0^3h_3h_5$ & $\{h_0^3h_3h_5\}$           &   Proposition \ref{easyones}    \\
  &   $4\rho_{15}$     &  $h_5Ph_1$ & $\{h_5Ph_1\}$            &    Proposition \ref{easyones}   \\
  &    $8\rho_{15}$    &  $h_5Ph_1^2$ & $\eta\{h_5Ph_1\}$             &  Proposition \ref{easyones}    \\
  &    $16\rho_{15}$    &  $h_0^2h_5Ph_2$  & $4\{h_5Ph_2\}$              &  Proposition \ref{easyones}    \\
  &    $\eta\kappa$    &   $h_2d_1$ &  $\nu\kappa_1$              &    Proposition \ref{easyones}   \\
 16   &  $\eta\rho_{15}$   &  $h_2t$ &  $\nu\{t\}$             &  Proposition \ref{easyones}     \\
 17& $\eta\eta_4$ &  $h_2^2h_5$ &  $\{h_2^2h_5\}$   &    Proposition \ref{easyones}   \\
  & $\eta^2\rho_{15}$ & $h_1^2g_2$ & $\eta^2\bar{\kappa}_2$  &      Proposition \ref{easyones}    \\ 
  &  $\nu\kappa$&  $h_3d_1$ & $\sigma\kappa_1$    &   Proposition \ref{easyones}   \\
  
  & $\mu_{17}$ & $h_1Ph_5c_0$ & $\eta\{Ph_5c_0\}$    &   Proposition \ref{easyones}     \\ 
18  & $2\nu_4$ & $h_1h_3h_5$ & $\sigma\eta_5$     &   Proposition \ref{easyones}  \\
  & $4\nu_4$ & $h_1^2h_3h_5$ & $\eta\sigma\eta_5$     &  Proposition \ref{easyones}   \\ 
   & $\eta\mu_{17}$ & $\triangle^2h_2^2$ & $\{\triangle h_1d_0^2\}$     & Proposition \ref{5.14}  \\
19    & $2\{P^2h_2\}$ & $h_0^2h_5i$ & $\{P^4h_0^2i\}$      &   Proposition \ref{5.15}  \\
  & $4\{P^2h_2\}$ & $h_1Q_1$ & $\{P^6c_0\}$     &        Proposition \ref{5.15}\\
  20  & $\bar{\kappa}$  & $g_2$ & $\bar{\kappa}_2$      &   Proposition \ref{easyones}   \\
  & $2\bar{\kappa}$ & $h_1g_2$ & $\eta\bar{\kappa}_2$     &    Proposition \ref{easyones}  \\ 
21  & $\eta\bar{\kappa}$ & $h_2g_2$ & $\nu\bar{\kappa}_2$      &    Proposition \ref{easyones}  \\
  & $\nu\nu_4$ & $h_4^3$ & $\theta_{4,5}$     &    Proposition \ref{easyones}  \\ 
22  & $\nu\bar{\sigma}$ & $h_3c_2$ & $\nu\bar{\kappa}_2$     &    Proposition \ref{5.17}  \\
& $\eta^2\bar{\kappa}$ & $d_1g$ & $\bar{\kappa}\kappa_1$     &    Proposition \ref{easyones}  \\
23  & $\nu\bar{\kappa}$ & $h_3g_2$ & $\sigma\bar{\kappa}_2$      &    Proposition \ref{easyones}  \\
  & $2\nu\bar{\kappa}$ & $h_1h_3g_2$ & $\eta\sigma\bar{\kappa}_2$     &  Proposition \ref{easyones} \\ 
  & $\rho_{23}$ & $h_0^7h_5^2$ & $\{\triangle^2 h_3^2\}$   &      Proposition \ref{5.16} \\
   &  $2\rho_{23}$ & $\triangle^2h_1h_4$ & $\{\triangle^2h_1h_4\}$     &    Proposition \ref{easyones}  \\ 
  &  $4\rho_{23}$ & $h_1\triangle^2h_1h_4$ & $\eta\{\triangle^2h_1h_4\}$     &    Proposition \ref{easyones}  \\
  & $8\rho_{23}$ & $h_1^2\triangle^2h_1h_4$ & $\eta^2\{\triangle^2h_1h_4\}$     &     Proposition \ref{easyones} \\ 
24  & $\eta\sigma\eta_4$ & $h_2h_5c_1$ & $\nu\{h_5c_1\}$      &   Proposition \ref{easyones}  \\
  & $\eta\rho_{23}$ & $\triangle^2 c_1$ & $\{\triangle^2 c_1\}$     &  Proposition \ref{easyones}   \\ 
25  & $\eta^2\rho_{23}$ & $h_2\triangle^2 c_1$ & $\nu\{\triangle^2 c_1\}$     &  Proposition \ref{easyones}   \\ 
  & $\mu_{25}$ & $\triangle^2h_2g$ & $\{\triangle^2h_2g\}$  &      Proposition \ref{easyones}  \\
26    & $\eta\mu_{25}$ & $h_2\triangle^2h_2g$ & $\nu\{\triangle^2h_2g\}$    &   Proposition \ref{easyones}   \\
      & $\mu^2\bar{\kappa}$ & $\triangle_1h_3^2$ & $\{\triangle_1h_3^2\}$    &    Proposition \ref{easyones}  \\
\end{longtable}

\begin{notn}
    Here all spectra are localized at the prime 2. The notations about elements in homotopy groups and $E_2$-page of the Adams spectral sequence are taken from Isaksen-Wang-Xu\cite{MR4588596}.
\end{notn}

\section{Preliminaries}
\subsection{Atiyah-Hirzebruch spectral sequence (AHSS) and $P^\infty_k$}

\hspace*{\fill}

In this subsection, we review the Adams spectral sequence (ASS) and the Atiyah-Hirzebruch spectral sequence (AHSS), and discuss the attaching maps in $P^\infty_k$.

Let \( X \) be a connective CW spectrum such that $H^*(X, \F_2)$ is of finite type. The mod $2$ Adams spectral sequence (ASS) for $X$ has $E_2$-term which converges strongly to the \( 2 \)-completion $({\pi_{t-s}(X^\wedge_2)})$:
\[
E_2^{s,t} = \operatorname{Ext}_{\mathcal{A}_*}^{s,t}(H^*(X; \mathbb{F}_2), \mathbb{F}_2))\Rightarrow ({\pi_{t-s}(X^\wedge_2)}) .
\]

Let \( X \) be a spectrum that is bounded below, which means \( \pi_q(X) = 0 \) for all \( q \) sufficiently small. Consider the skeletal filtration of \( X \):
\[
\emptyset = X^{-k} \subset X^{-k+1} \subset X^{-k+2} \subset \cdots \subset X^n \subset X,
\]
The long exact sequences
\[
\cdots \rightarrow \pi_{p+q}X^{p-1} \xrightarrow{i} \pi_{p+q}X^p \xrightarrow{j} \pi_{p+q}(X^p/X^{p-1}) \xrightarrow{k} \pi_{p+q-1}X^{p-1} \rightarrow \cdots
\]
yield the Atiyah-Hirzebruch spectral sequence (AHSS), whose \( E_1 \)-page is given by:
\[
E_1^{s,t} = \pi_t(X^s/X^{s-1}).
\]

We assume that \( X \) has at most one cell in each dimension. Under this assumption, any element in the \( E_1 \)-page can be denoted as \( \alpha[s] \), where \( \alpha \) is an element in the stable homotopy group of spheres and \( s \) is its Atiyah-Hirzebruch filtration. For simplicity, we will use the same notation \( \alpha[s] \) to represent an element in \( \pi_*(X) \).

The differential
\[
d_r: E_r^{s,t} \rightarrow E_r^{s-r,t-1}
\]
is defined via the attaching map. Let \( \tilde{\alpha} \) be an element in \( \pi_t(X^s/X^{s-r}) \) that maps to \( \alpha[s] \) under the projection map \( X^s/X^{s-r} \twoheadrightarrow X^s/X^{s-1} \). Then, \( d_r(\alpha[s]) \) is defined as the composition of \( \tilde{\alpha} \) with the attaching map \( X^s/X^{s-r} \rightarrow \Sigma X^{s-r}/X^{s-r-1} \).
\[
\begin{tikzcd}
S^t \arrow[r, "\tilde{\alpha}"] & X^s/X^{s-r} \arrow[r] & \Sigma X^{s-r}/X^{s-r-1}
\end{tikzcd}
\]

Our computations of Mahowald invariant of these low stems are a combination of the AHSS of $P^{\infty}_{-N}$ and the cell structures of $P^{\infty}_{-N}$.

The following theorem tells us the periodicity of the cell structures of $P^{\infty}_{-N}$.
\begin{theorem}
[James periodicity, \cite{MR177411}] 
\[
P^{n-1}_{n-r-1}\simeq\Sigma^{f(r)}P^{n-f(r)-1}_{n-r-f(r)-1}
\]
where $f(r)=2^{g(r)}$ and $g(r)=\lfloor\frac{r}{2}\rfloor+\left\{
\begin{array}{rcl}
-1 & & r\equiv 0\enspace \mathrm{mod} \thinspace 8\\
1 & &  r\equiv 3,5\enspace \mathrm{mod} \thinspace 8\\
0 & & else
\end{array}
\right. $
\end{theorem}

By the Steenrod squares on $P^{\infty}_k$ for any positive integer $k$, we have the following proposition:
\begin{proposition}\cite{MR836132}
In $P^{\infty}_k$, there is an attaching map $2\iota$ from $(n+1)$ cell to $n$ cell for $n>k$ if and only if $n\equiv 1 \thinspace(\mathrm{mod}\thinspace 2)$, there is an attaching map $\eta$ from $(n+2)$ cell to $n$ cell for $n>k$ if and only if $n\equiv 2,3 \thinspace(\mathrm{mod}\thinspace 4)$, there is an attaching map $\nu$ from $(n+4)$ cell to $n$ cell for $n>k$ if and only if $n\equiv 4,5,6,7 \thinspace(\mathrm{mod}\thinspace 8)$ and there is an attaching map $\sigma$ from $(n+8)$ cell to $n$ cell for $n>k$ if and only if $n\equiv 8,9,10,11,12,13,14,15 \thinspace(\mathrm{mod}\thinspace 8)$.  
\end{proposition}

\subsection{Algebraic Mahowald Invariant and Filtered Mahowald Invariant}

\hspace*{\fill}

In this subsection, we review the definition of algebraic Mahowald invariants in \cite{MR1241877} and $E$-filtered Mahowald invariants in \cite{behrens2005rootinvariantsadamsspectral}. In Behrens\cite{behrens2005rootinvariantsadamsspectral}, he proved the relation between the $HF_p$-filtered Mahowald invariants and the algebraic Mahowald invariants and introduced the differential of $E$-filtered Mahowald invariants on $P^{\infty}_{-N}$.

\begin{definition}
Let \( \alpha \) be an element of \( \Ext^{s,t}(H_*X) \). The algebraic Mahowald invariant \( M_{alg}(\alpha) \) is defined by the following diagram of \( \Ext \) groups:
\[
\begin{tikzcd}
{\Ext^{s,t}(H_*X)} \arrow[rr, "M_{alg}(\alpha)", dashed] \arrow[d, "i_{\#}"'] & & {\Ext^{s,t+N-1}(H_*X)} \arrow[d, "\iota_N"] \\
{\Ext^{s,t-1}(H_*P^{\infty}_{-\infty} \wedge X)} \arrow[rr, "\nu_N"'] & & {\Ext^{s,t-1}(H_*P^{\infty}_{-N} \wedge X)} 
\end{tikzcd}
\]
Here, \( i_* \) is induced by the inclusion of the \(-1\)-cell of \( P^{\infty}_{-\infty} \), \( \nu_N \) is the projection onto the \(-N\)-coskeleton, \( \iota_N \) is the inclusion of the \(-N\)-cell, and \( N \) is minimal such that \( \nu_N \circ i_*(\alpha) \) is zero. The algebraic Mahowald invariant is defined as the coset of lifts \( \gamma \in \Ext^{s,t+N-1}(H_*X) \) of the element \( \nu_N \circ i_*(\alpha) \).
\end{definition}

We assume that \( u \) is a nontrivial permanent cycle in the \( E_2 \)-page and detects the homotopy map \( f \). However, \( M_{alg}(u) \) may fail to contain a permanent cycle. Consider the following diagram of \( \Ext \) groups:
\begin{equation}\label{2}
\begin{tikzcd}
E_2(S^{t-1}) \arrow[r, "v"] \arrow[d, "u"'] & E_2(S^{-n}) \arrow[d, "i_*"'] & \\
E_2(S^{-1}) \arrow[r, "h_{\#}"] & E_2(P^{\infty}_{-n}) \arrow[r, "j_{\#}"] & E_2(P^{\infty}_{-n+1})
\end{tikzcd}.
\end{equation}
Suppose $n$ is not the smallest with respect to the property that $h\circ f$ is nontrivial, which can happen when $j\circ h\circ f$ is essential but has higher Adams filtration than expected. In this case, we can't get the following commutative diagram:
\begin{equation}\label{1}
\begin{tikzcd}
S^{-1+t} \arrow[r, "g"] \arrow[d, "f"'] & S^{-n} \arrow[d, "i"']         &                   \\
S^{-1} \arrow[r, "h"]                   & P^{\infty}_{-n} \arrow[r, "j"] & P^{\infty}_{-n+1}
\end{tikzcd}
\end{equation}
and $M_{alg}(u)$ does not contain a permanent cycle, and the homotopy Mahowald invariant $M(f)$ has smaller stems than the algebraic Mahowald invariant $M_{alg}(u)$.

The following theorem demonstrates the relation between algebraic Mahowald invariants and homotopy Mahowald invariants.

\begin{theorem}\label{easy}\label{3.4}\cite[Theorem 2.9]{MR1241877}
Let \( f \in \pi_t(S^0) \) be a nontrivial homotopy element representing a class \( u \in E_2(S^0) \), and suppose that the algebraic Mahowald invariant \( M_{alg}(u) \) lies in dimension \( k \).
\begin{enumerate}
    \item If \( M_{alg}(u) \) does not contain a permanent cycle, then the dimension of \( M(f) \) is less than \( k \).
    \item If the diagram \ref{1} exists but \( h_{\#}(u) \) is killed by a differential, then \( M(f) \) has the same stem but higher Adams filtration than \( M_{alg}(u) \).
    \item If the diagram \ref{1} exists and \( h_{\#}(u) \) is nontrivial in the \( E_{\infty} \)-page, then \( M(f) \) is contained in the homotopy coset representing \( M_{alg}(u) \).
    \item If the diagram \ref{1} exists and the map \( hf \) is null, then the dimension of \( M(f) \) is greater than \( k \).
\end{enumerate}
\end{theorem}

We now recall the definition of filtered Mahowald invariants as given by Behrens \cite{behrens2005rootinvariantsadamsspectral}. Let \( E \) be a ring spectrum for which the \( E \)-Adams spectral sequence converges, and let \( \bar{E} \) be the fiber of the unit map \( S \rightarrow E \). The \( E \)-Adams resolution of the sphere is given by:
\[
\begin{tikzcd}
S^0 & W_0 \arrow[d] \arrow[l, no head, Rightarrow] & W_1 \arrow[d] \arrow[l] & W_2 \arrow[d] \arrow[l] & W_3 \arrow[d] \arrow[l] & \cdots \arrow[l] \\
    & Y_0 & Y_1 & Y_2 & Y_3 & 
\end{tikzcd}
\]
where $W_k=\bar{E}^{(k)}$ and $Y_k=E\wedge\bar{E}^{(k)}$. Together with the skeletal filtration of $P^{\infty}_{-\infty}$, we may regard $P^{\infty}_{-\infty}$ as a bifiltered object with $(k,N)$-bifiltration given by \[
W_k(P^N)=(W_k\wedge P^N)_{-\infty}
\]

The spectra $W_k$ may be replaced by weakly equivalent approximations so that
for every $k$ the map $W_{k+1}\rightarrow W_k$ are inclusions of subcomplexes. Then we know that for $k_1\geq k_2$ and $N_1\leq N_2$ the bifiltration $W_{k_1}(P^{N_1})$ is a subcomplex of $W_{k_2}(P^{N_2})$.

Given sequences\[
I=\{k_1<k_2<\cdots<k_l\}
\]
\[
J=\{-N_1<-N_2<\cdots<-N_l\}
\]
with $k_i\geq0$, the filtered Tate spectrum is defined as the union
\[
W_I(P^J)=\mathop{\cup}\limits_{i}W_{k_i}(P^{N_i}),
\]
and for $1\leq i\leq l$, we have natural projection maps:\[
p_i:W_I(P^J)\rightarrow W_{k_i}(P^{N_i})
\]
by smashing with $E$. The filtered Mahowald invariants are defined as follows.

\begin{definition}\label{3.5}\cite{behrens2005rootinvariantsadamsspectral}
Let $\alpha$ be an element of $\pi_t(S)$, with image $l(\alpha)\in\pi_{t-1}(P^{\infty}_{-\infty})$. Choose a multi-index $(I,J)$ where $I=(k_1,k_2,\cdots)$ and $J=(N_1,N_2,\cdots)$ so that the filtered Tate spectrum $W_I(P^J)$ is initial amongst the Tate spectra $W_K(P^L)$ so that $l(\alpha)$ is in the image of the map
\[
\pi_{t-1}(W_K(P^L))\rightarrow \pi_{t-1}(P^{\infty}_{-\infty})
\]
Let $\widetilde{\alpha}$ be a lift of $l(\alpha)$ to $\pi_{t-1}(W_I(P^J))$. Then the $k_i$th $E$-filtered Mahowald invariant is given by\[
M_E^{[k_i]}(\alpha)=p_i(\widetilde{\alpha})\in\pi_{t-1}(Y_{k_i}\wedge S^{N_i}).
\]

To explain the property ``initial" precisely, we define $S(I,J):=\mathop{\cup}\limits_{i=1}^{l}\{(a,b):a\geq k_i,b\leq N_i\}$, and $(I',J')\leq(I,J)$ if and only if $S(I', J')\subseteq S(I,J)$. Given two pairs of sequences $(I',J')\leq(I,J)$, we define spectra\[
 W^{I'}_I(P^J_{J'})=cofiber(W_{I'+1}(P^{J'-1}\rightarrow W_I(P^J))
 \]
where $I'+1$(respectively $J'-1$) is the sequence obtained by increasing(decreasing) every element of the sequence by $1$.

We shall define a pair of sequence $(I,J)$ associated to $\alpha$ inductively. Let $k_1$ be maximal such that the composite\[
\begin{tikzcd}
S^{t-1} \arrow[r, "\alpha"] & \Sigma^{-1}X \arrow[r] & \Sigma^{-1}tX \arrow[r] & W^{k_1-1}_0(P\wedge X)_{-\infty}
\end{tikzcd}
\]is trivial. Here $tX$ is the Tate spectrum of $X$. Next, choose $N_1$ to be maximal such that the composite\[
\begin{tikzcd}
S^{t-1} \arrow[r, "\alpha"] & \Sigma^{-1}X \arrow[r] & \Sigma^{-1}tX \arrow[r] & {W^{(k_1-1,k)}_0(P_{(-N_1+1,\infty)}\wedge X)}
\end{tikzcd}
\]is trivial. Inductively, given $I'=(k_1,k_2,\cdots,k_i)$ and $J'=(-N_1,-N_2,\cdots,-N_i)$, let $k_{i+1}$ be maximal so that the composite\[
\begin{tikzcd}
S^{t-1} \arrow[r, "\alpha"] & \Sigma^{-1}X \arrow[r] & \Sigma^{-1}tX \arrow[r] & {W^{(I'-1,k_{i+1}-1)}_0(P_{(J'+1,\infty)}\wedge X)}
\end{tikzcd}
\]is trivial. If there is no such $k_{i+1}$, we declare that $k_{i+1}=\infty$ and finish the induction. Otherwise, choose $N_{i+1}$ to be maximal such that the composite\[
\begin{tikzcd}
S^{t-1} \arrow[r, "\alpha"] & \Sigma^{-1}X \arrow[r] & \Sigma^{-1}tX \arrow[r] & {W^{(I'-1,k_{i+1}-1,k_{i+1})}_0(P_{(J'+1,-N_{i+1}+1,\infty)}\wedge X)}
\end{tikzcd}
\]is trivial, and continue the inductive procedure.
\end{definition}

Similarly, there is an indeterminacy in the filtered root invariants based on the choice of $\widetilde{\alpha}$.

The relations among homotopy Mahowald invariants, filtered Mahowald invariants and algebraic Mahowald invariants are explained in the following theorems.

\begin{theorem}\label{3.6}\cite[Theorem 5.1]{behrens2005rootinvariantsadamsspectral}
Suppose that $R^{[k_i]}_E(\alpha)$ contains a permanent cycle $\beta$. Then there exists an element $\bar{\beta}\in\pi_*(X)$ detected by $\beta$ such that the following diagram commutes up to elements of $E$-Adams filtration greater than or equal to $k_{i+1}$:
\[
\begin{tikzcd}
S^t \arrow[r, "\bar{\beta}"'] \arrow[d, "\alpha"] & \Sigma^{-N_i+1}X \arrow[dd] \\
X \arrow[d]                                       &                             \\
tX \arrow[r]                                      & \Sigma P_{-N_i}\wedge X    
\end{tikzcd}
\]
\end{theorem}

\begin{corollary}\label{3.7}\rm{(\cite{Behrens_2007}, Corollary 6.2)}
 Let $\beta$ be the element described in Theorem \ref{3.6}. Then in order for $\beta$ to detect the homotopy Mahowald invariant in the $E$-ASS, it is sufficient to check two things:
 \begin{enumerate}
     \item No element $\gamma\in\pi_{t-1}(P_{-N_i})$ of $E$-Adams filtration greater than $k_i$ can detect the Mahowald invariant of $\alpha$ in $P_{-N_i+1}$.
     \item The image of the element $\bar{\beta}$ under the inclusion of the bottom cell\[
     \pi_{t-1}(S^{-N_i})\rightarrow\pi_{t-1}(P_{-N_i})
     \]is nontrivial.
 \end{enumerate}
\end{corollary}

Before the next theorem, we recall the definition of $K$-Toda bracket by Behrens\cite{Behrens_2007}: Let $K$ be a finite CW complex with a single cell in top dimension $n$ and the bottom dimension $0$. There is an inclusion map $\iota:S^0\rightarrow K$ and the $n$-cell is attached to the $n-1$ skeleton $K^{n-1}$ by an attaching map $a:S^{n-1}\rightarrow K^{n-1}$. Let $\beta$ be an element in $\pi_t(S)$, then the $K$-Toda bracket is defined to a lift in the following diagram:\[
\begin{tikzcd}
S^{t+n-1} \arrow[r, "\beta"] \arrow[rrd, "\langle K\rangle(\beta)", dashed] & S^{n-1} \arrow[r, "a"] & K^{n-1}                      \\
                                                                            &                        & S^0 \arrow[u, "\iota", hook]
\end{tikzcd}.
\]
\begin{theorem}\label{3.8}\cite[Theorem 5.3]{behrens2005rootinvariantsadamsspectral}
Suppose that the $P^{-N_{i+1}}_{-N_i}$-Toda bracket has $E$-Adams degree $d$ and $d\leq k_{i+2}-k_{i+1}$.Then the following statements are true:
\begin{enumerate}
    \item $\langle P^{-N_{i+1}}_{-N_i}\rangle(R^{[k_{i+1}]}_E(\alpha))$ is defined and contains a permanent cycle.
    \item $R^{[k_i]}_E(\alpha)$ consists of elements which are $d_r$ cycles for $r<k_{i+1}-k_i+d$.
    \item There is a containment\[
    d_{r_i+d}R^{[k_i]}_E(\alpha)\subseteq\langle P^{-N_{i+1}}_{-N_i}\rangle(R^{[k_{i+1}]}_E(\alpha))
    \]where elements of both sides are thought of as elements of $E^{*,*}_{k_{i+1}-k_i+d}$.
\end{enumerate}
\end{theorem}

\begin{theorem}\label{3.9}\cite[Theorem 5.10]{behrens2005rootinvariantsadamsspectral}

If $E$ is the Eilenberg-MacLane spectrum $HF_p$ and $\alpha$ has Adams filtration $k$, then $k_1=k$, Furthermore, the filtered Mahowald invariant $M_{HF_p}^{[k]}$ consists of $d_1$ cycles which detect a coset of non-trivial elements $\bar{R}^{[k]}_{HF_p}(\alpha)\subseteq E_2^{k,t+k+N_1-1}(X)$ and there exists a choice of $\widetilde{\alpha}\in E_2^{t,t+k}(X)$ which detects $\alpha$ in the ASS such that\[
\bar{R}^{[k]}_{HF_p}(\alpha)\subseteq M_{alg}(\widetilde{\alpha})
\].
    
\end{theorem}

To compute algebraic Mahowald invariants, we need the assistance of squaring operations in $\Ext(\F_2,\F_2)$ constructed by Milgram:
\begin{proposition}\label{3.10}\cite[Theorem 3.1.3 and Theorem 4.1.1]{MR304463}
There are operations $Sq^i$ in $\Ext_{\mathcal{A}_*}(\F_2,\F_2)$ so that
$$
d_2(Sq^i(a))=\left\{
\begin{array}{l}
  h_0Sq^{i+1}(a)\quad i\equiv t\enspace(\mathrm{mod}\thinspace2)
  \\0\quad otherwise
\end{array}\right.
$$

for $a\in\Ext_{\mathcal{A}_*}^{s,t}(\F_2,\F_2)$
\end{proposition}
    
 In this setting, $Sq^0$ is not the identity but in general is a non-zero class in twice the $t$-filtration but in the same $s$-filtration, so we deduce that\[
 Sq^0(h_i)=h_{i+1}, \forall i\in\N
 \]
More generally, we know $Sq^0(x)\in M_{alg}(x)$ if $Sq^0(x)\neq0$ by \cite[Proposition 2.5]{MR1241877}, by which Bruner gives the results of algebraic Mahowald invariants in $\Ext$ over Steenrod algebra through the 25-stem in Bruner\cite{MR1642887}. Since there is a Cartan formula on square operations, we obtain the following corollary:
\begin{corollary}\label{3.11}
  If $a$ and $b$ are two elements in $\Ext(S^0)$ with $Sq^0(a)Sq^0(b)\neq0$, then $Sq^0(ab)=Sq^0(a)Sq^0(b)\in M_{alg}(ab)$.   
\end{corollary}

The equivariant definition of Mahowald invariants provides an elementary proof of the Cartan formula in the homotopy Mahowald invariant, which will be used in the computations of homotopy Mahowald invariants later:
\begin{theorem}\label{3.12}\cite[Theorem 1]{MR1622335}
 Let $\alpha_i\in\pi_{n_i}(S^0)$ and $M(\alpha_i)\in\pi_{n_i+k_i}(S^0)$ for $i=1,2$. Let $k=k_1+k_2$ and let $i:S^{-k-1}\rightarrow P^{\infty}_{-k-1}$ be the inclusion of the bottom cell of the stunted projective space $P^{\infty}_{-k-1}$. Then we have:
 \begin{enumerate}
     \item If $i_*(M(\alpha_1)M(\alpha_2))\neq0$, then $M(\alpha_1)M(\alpha_2)\subset M(\alpha_1\alpha_2)$
     \item If $i_*(M(\alpha_1)M(\alpha_2))=0$, then $M(\alpha_1\alpha_2)$ lies in a higher stem than does $M(\alpha_1)M(\alpha_2)$.
 \end{enumerate}
\end{theorem}

\section{Computations of homotopy Mahowald invariants}
After the preparations above, we start our computations of the homotopy Mahowald invariants of elements up to 26-stem from the elements whose algebraic Mahowald invariants are nontrivial in the $E_\infty$-page. 
All the information about algebraic Mahowald invariants is from Bruner\cite{MR1642887} and Corollary \ref{3.11}. The notations and the data about Adams differentials and hidden extensions are taken from \cite{MR4588596}.

\begin{proposition}\label{easyones}
    For those elements in Table \ref{table} whose algebraic Mahowald invariants are nontrivial in the $E_\infty$-page, their homotopy Mahowald invariants are precisely the corresponding elements of the algebraic Mahowald invariants in the $E_\infty$-page.
\end{proposition}

\begin{proof}
    This can be obtained directly by Theorem \ref{easy}.
\end{proof}

\hspace*{\fill}

For some other elements in Table \ref{table}, we follow Procedure 9.1 of Behrens\cite{behrens2005rootinvariantsadamsspectral} to check every candidate through information about the stem and filtration. These homotopy Mahowald invariants may have indeterminacy.

\begin{proposition}\label{5.14}
$\{\triangle h_1d_0^2\}\in M(\eta\mu_{17})$ 
\end{proposition}
\begin{proof}
 The element $\eta\mu_{17}$ is detected by $h_1P^2h_1$ in the ASS, and $\triangle^2h_2^2\in M_{alg}(h_1P^2h_1)$ with indeterminacy $h_0^2h_5i$. By Theorem \ref{3.9} we know $\triangle^2h_2^2\in R^{[10]}_{HF_2}(\eta\mu_{17})$. We have the Adams differential $d_2(\triangle^2h_2^2)=h_0^2PM$, and we have $\langle P^{-36}_{-37}\rangle h_0PM=h_0^2PM$ by Theorem \ref{3.8}, so we know $h_0PM\in R^{[11]}_{HF_2}(\eta\mu_{17})$ by Theorem \ref{3.8}. By Theorem \ref{3.4}(a) we know $|M(\eta\mu_{17})|\leq53$. Since 
 \[\eta\{Ph_5c_0\}\in M(\mu_{17}), \nu\in M(\eta) \;\mathrm{and}\; \mu\eta\{Ph_5c_0\}=0,\] 
we know $|M(\eta\mu_{17})|\geq52$ by Theorem \ref{3.12}. So by Theorem \ref{3.6} it suffices to check the generators of $\pi_{52}$ and $\pi_{53}$ with filtrations greater than $11$. 
 
Since there is only one element $\{\triangle h_1d_0^2\}$ satisfying these conditions and $h_0PM$ is killed through the $d_3$ Adams differential\[
d_3(h_5i)=h_0PM,\] we deduce that $\{\triangle h_1d_0^2\}\in M(\eta\mu_{17})$.
\end{proof}

\begin{proposition}\label{5.15}
$\{P^4h_0^2i\}\in M(2\{P^2h_2\})$ and $\{P^6c_0\}\in M(4\{P^2h_2\})$
\end{proposition}

\begin{proof}
 We know that $h_5Pd_0\in M_{alg}(P^2h_2)$ and that $h_5Pd_0$ is killed through the $d_3$ Adams differential\[
d_3(\triangle_1h_1^2)=h_5Pd_0,\] 
so by Theorem \ref{3.4} we know $|M(\{P^2h_2\})|\geq53$. By Theorem \ref{3.12} and that $M(2)=\eta$, we know that $|M(4\{P^2h_2\})|\geq55$. Since $h_1\triangle^2h_1h_3\in M_{alg}(h_0^2P^2h_2)$ and $h_1\triangle^2h_1h_3$ supports the $d_2$ Adams differential, by Theorem \ref{3.4}(a) we know that $|M(4\{P^2h_2\})|\leq56$. 
 
By Theorem \ref{3.12}, we know that\[ |M(\{P^2h_2\})|+1\leq |M(2\{P^2h_2\})|\leq |M(4\{P^2h_2\})|-1,\]
so there is an equality $|M(\{P^2h_2\})|+1=|M(2\{P^2h_2\})|$ or $|M(2\{P^2h_2\})|+1=|M(4\{P^2h_2\})|$. So by Theorem \ref{3.12} there is an $\eta$-extension from $M(\{P^2h_2\})$ to $M(2\{P^2h_2\})$ or from $M(2\{P^2h_2\})$ to $M(4\{P^2h_2\})$. 

In the AHSS of $P^\infty_{-36}$ we have $d_4(\nu\{Mh_2\}[-32])=\eta\{PM\}[-36]$ because there is a hidden $\nu$-extension from $\nu\{Mh_2\}$ to $\eta\{PM\}$, so we know that the only possibility is the hidden $\eta$-extension from $\{P^4h_0^2i\}$ to $\{P^6c_0\}$, which means $\{P^4h_0^2i\}\in M(2\{P^2h_2\})$ and that $\{P^6c_0\}\in M(4\{P^2h_2\})$.
\end{proof}

\begin{proposition}\label{5.16}
$\{\triangle^2h_3^2\}\in M(\rho_{23})$    
\end{proposition}
\begin{proof}
The element $\rho_{23}$ is detected by $h^2_0i+ h_1Pd_0$ in the ASS, and $h_0^7h_5^2\in M_{alg}(h^2_0i+ h_1Pd_0)$. Since $h_0^7h_5^2$ is killed through the $d_2$ Adams differential, we know $|M(\rho_{23})|\geq62$.
We know $|M(\rho_{23})|\leq |M(2\rho_{23})|-1$ by Theorem \ref{3.12}. 

From Proposition \ref{5.15} we know $\{\triangle^2h_1h_4\}\in M(2\rho_{23})$, and there is no $\eta$-extension from 63-stem to $\{\triangle^2h_1h_4\}$, so we can deduce that $|M(\rho_{23})|\leq62$, and that $|M(\rho_{23})|=62$. By Theorem \ref{3.9} we know $h_0^7h_5^2\in R^{[9]}_{HF_2}(\rho_{23})$. By Theorem \ref{3.6} it suffices to check the nontrivial elements of 62-stem with filtration greater than 9. The only possibility is $\{\triangle^2h_3^2\}\in M(\rho_{23})$.

\end{proof}

\begin{proposition}\label{5.17}
 $\nu\bar{\kappa}_2\in M(\nu\bar{\sigma})$   
\end{proposition}
\begin{proof}
The element $\nu\bar{\sigma}$ is detected by $h_2c_1$ in the ASS, and $h_3c_2\in M_{alg}(h_2c_1)$. By Theorem \ref{3.9} we know $h_3c_2\in R^{[4]}_{HF_2}(\nu\bar{\sigma})$. We have the Adams differential $d_2(h_3c_2)=h_0h_2g_2$, and by Theorem \ref{3.8} we have $\langle P^{-26}_{-27}\rangle h_3g_2=h_0h_3g_2$, so by Theorem \ref{3.8} we know $h_2g_2\in R^{[5]}_{HF_2}(\nu\bar{\sigma})$. By $|\nu\bar{\sigma}|=22$ and Theorem \ref{3.4}(a), we deduce $44\leq|M(\nu\bar{\sigma})|\leq 47$. By Theorem \ref{3.6} it suffices to check the generators of $\pi_{44}$,$\pi_{45}$ and $\pi_{46}$ with filtrations greater than 5.

Since $|M(\bar{\sigma})|\geq 38$, by Theorem \ref{3.12} we have \[M(\nu\bar{\sigma})\geq 38+|M(\nu)|=45.\]
In the AHSS of $P^\infty_{-25}$, we have Atiyah-Hirzebruch differentials \[\begin{split}
   & d(\{h_5d_0\}[-23])=\eta\{h_5d_0\}[-25],d(\eta\bar{\kappa}_2[-23])=\eta^2\bar{\kappa}_2[-25],\\
   & d(\theta_{4,5}[-23])=\{Mh_1\}[-25]\; \mathrm{and}\; d(\{\triangle h_1g\}[-23])=\{d_0l\}[-25].
\end{split} \] 
They are all candidates of 46-stem, so $|M(\nu\bar{\sigma})|\neq 46$. 

If $|M(\nu\bar{\sigma})|=45$, we have \[M(\bar{\sigma})+M(\nu)\geq 38+7=45,\] so we have 
$M(\nu)M(\bar{\sigma})\subseteq M(\nu\bar{\sigma})$. Similarly, we know $c_2\in M_{alg}(c_1)$, so $f_1\in R^{[4]}_{HF_2}(\bar{\sigma})$. By Theorem \ref{3.6}, the only possibility of $M(\bar{\sigma}) $ with a $\sigma$-extension is $\{h_0^3h_3h_5\}$. However, in the AHSS of $P_{-24}$ we have an Atiyah-Hirzebruch differential 
\[d(\{h_0^3h_3h_5\}[-16])=8\theta_{4,5}[-24],\] so $\sigma\{h_0^3h_3h_5\}\notin M(\nu\bar{\sigma})$, which is a contradiction.

Therefore, all candidates of 44-stem to 46-stem are impossible and we deduce $\nu\bar{\kappa}_2\in M(\nu\bar{\sigma})$.
\end{proof}

\begin{remark}
By these methods, there are still some elements up to 26-stem whose homotopy Mahowald invariants are unknown, although some possibilities can be excluded. The difficulty of the computations is to find the top filtered Mahowald invariant.
\end{remark}

\section*{Acknowledgments}
We thank Zhouli Xu for providing the opportunity to conduct this research and for his supervision. We are deeply indebted to Shangjie Zhang for his mentorship and many illuminating discussions on stable homotopy groups. We thank Mark Behrens for his guidance, and Robert Bruner and Dan Isaksen for identifying errors in an earlier draft and offering suggestions that clarified our results.
\bibliographystyle{alpha}
\bibliography{reference}

\end{document}